\documentclass[preprint, 11pt]{amsart}
\usepackage{amsfonts,amssymb,amsmath,amscd,mathtools,multicol,tikz, tikz-cd,caption,enumerate,mathrsfs,geometry,tgpagella}
\usepackage{amsthm}
\usetikzlibrary{arrows,chains,matrix,positioning,scopes}
\usepackage{bigfoot}
\usepackage[pagebackref=true, colorlinks, linkcolor=blue, citecolor=blue]{hyperref}
\usepackage{fullpage}

\usepackage{multirow}
\usepackage[multiple]{footmisc}
\usepackage{kerkis}
\makeatletter
\makeindex
\newtheorem{theorem}{Theorem}[section]
\newtheorem{proposition}[theorem]{Proposition}
\newtheorem*{propostion*}{Proposition*}
\newtheorem*{theorem*}{Theorem}
\newtheorem{corollary}[theorem]{Corollary}
\newtheorem{definition}[theorem]{Definition}
\newtheorem{lemma}[theorem]{Lemma}
\newtheorem{remark}[theorem]{Remark}
\theoremstyle{remark}
\newtheorem{example}[theorem]{Example} 
\newtheorem{examples}[theorem]{Examples}
 
\newtheorem*{conjecture*}{\bf{Conjecture}}

\newtheorem{algorithm}{Algorithm}[section]
\newtheorem*{algorithm*}{Algorithm}
\numberwithin{equation}{section}

\begin{document}

	\title{\boldmath{New and General Type Meromorphic  $1$-forms on Curves}}
	\author{Partha Kumbhakar}
	\address{Indian Institute of Science Education and Research (IISER) Mohali Sector 81, S.A.S. Nagar, Knowledge City, Punjab 140306, India.}
	\email{\url{parthakumbhakar@iisermohali.ac.in}}

	\maketitle
	
	\begin{abstract}
		In this article, we study the existence of new and general type meromorphic $1$-forms on curves through explicit construction. Specifically, we have constructed a large family of new and general type meromorphic $1$-forms on $\mathbb{P}^1,$ elliptic and hyperelliptic curves. We also established a connection to the Hurwitz realization problem of branch cover for the Riemann Sphere, which provides an algorithm to determine whether a $1$-form on $\mathbb{P}^1$ (of some restricted class) is new or old.        
	\end{abstract}

	\section{Introduction}\label{Intro}
	Let $C$ be an algebraically closed differential field of characteristic zero with trivial derivation (i.e. $c'=0$ for all $c\in\ C$). Let $f\in C[Z,W]$ be an irreducible polynomial such that $\frac{\partial f}{\partial W}\neq 0,$ with function field $$C(z,w):=\mathrm{Frac}\left( \frac{C[Z,W]}{(f)}\right).$$ Consider a first order algebraic differential equation \begin{equation}
		f(u,u')=0,
	\end{equation}
	where $u'$ stands for the derivative of $u.$ Then $C(z,w)$ is a differential field with derivation $\delta$ defined by $\delta(z)=w$ and $\delta(c)=c'=0$ for all $c\in C.$ Let  $(X,\omega)$ be the pair where $X$ is the smooth projective curve with function field $C(X):=C(z,w)$ and $\omega\in \Omega^1_{X}$ is a meromorphic $1$-form on $X$ dual to the $C-$linear derivation $\delta.$  The $1$-form  $\omega$ and the derivation $\delta$ are related as follows: $\omega=dh/\delta(h)$ for any $h\in C(z,w)\setminus C.$ Moreover, any such pair $(X, \omega)$ with $\omega\neq0$ can be thought of as a first order algebraic differential equation over $C$ (\cite[Lemma 5.2]{NNvdP15}). As an example, the associated pair for the equation $u'-g(u)=0$ is $(\mathbb{P}^1, \frac{dx}{g(x)}).$
	
	To understand the algebraic dependency of the solutions of first order differential equations, the authors in \cite{NvdPT22} study certain geometric properties of the associated pairs. They classify the pairs (hence first order differential equations) into the following types:
	\begin{enumerate}[(i)]
		\item $(X,\omega)$ where $\omega=dh$ for some $h\in C(X),$ is called of \emph{exact type};  
		\item $(X,\omega)$ where $\omega=dh/ch$ for some $h\in C(X)$ and $c\in C\setminus 0,$ is called of \emph{exponential type}; 
		\item $(X,\omega)$ where $\omega=dh/g$ for some $g,h\in C(X)$ where $g^2=h^3+ah+b$ for $a,b \in C$ with $4a^3+27b^2\neq 0,$ is called of \emph{Weierstrass type}; 
		\item In the rest of the cases, $(X,\omega)$ is called of \emph{general type}. 
	\end{enumerate}

	Furthermore, they define a pair $(X,\omega)$ to be \emph{new}\footnote{In \cite{HI03}, the term \emph{essential} is used in place of \emph{new}.} if $(X, \omega)$ does not have a proper pullback, i.e., there exists no pair $(Y,\eta)$ and a morphism $\phi:X\rightarrow Y$ with $\mathrm{deg}(\phi)\geq 2$ such that $\phi^*\eta=\omega.$ Otherwise, $(X,\omega)$ is called \emph{old}.

	In the context of function fields, a pair $(X, \omega)$ is of exact type (respectively of exponential type, respectively of Weierstrass type) if an only if the differential field $C(X)$ contains an element $h, h\notin C$ satisfying $\delta(h)=1$ (respectively $\delta(h)=ch, c\in C\setminus 0,$ respectively $\delta(h)^2=h^3+ah+b$ where $a, b\in C$ with $4a^3+27b^2\neq 0$). A pair $(X, \omega)$ is new if and only if the differential field $C(X)$ is simple, i.e., if $K$ is a differential subfield of $C(X)$ with $C\subseteq K\subseteq C(X)$ then either $K=C$ or $K=C(X).$
	
	In \cite[Theorem 2.1(a)]{NvdPT22} it is proved that any number of distinct nonconstant solutions of an equation of general type and new are $C-$algebraically independent. The converse also holds: If any two nonconstant solutions of a first order algebraic differential equation over $C$ are $C-$algebraically independent, then the equation must be of general type and new.

	Furthermore, the equations of general type and new are "irreducible first order differential equations," and the nonconstant solutions of these equations are "new functions" in the sense that such solutions cannot be expressed iteratively in terms of solutions of  linear differential equations, abelian functions, or solutions of any other first order differential equations (cf. \cite{UH88,NK88}). Also, new and of general type equations do not possess the Painlevé property.

	In the publications \cite{HI03} and \cite{NvdPT22}, the existence of new $1$-forms on a curve \footnote{In the paper, we sometimes say \emph{a $1$-form $\omega$ on a curve $X$ is $\dots$} type, which will mean \emph{the pair $(X,\omega)$ is of $\dots$} type. Here, $\dots$ stands for the adjectives: exact, exponential, Weierstrass, general. Also, we say a $1$-form $\omega$ on a curve $X$ is new to mean that the pair $(X,\omega)$ is new. The same for old.} was considered. For  $\mathbb{P}^1,$ Hrushovski-Itai prove the following result (\cite[Lemma 2.23]{HI03}): Let $\omega\in \Omega^1_{\mathbb{P}^1}$ be a $1$-form such that $\omega$ has at least two nonzero residues and that no two distinct nonzero residues of $\omega$ are linearly dependent over $\mathbb{Q},$ then $\omega$ is new. Their proof can be extended to show that $\omega$ is also general type. This result can be used to construct new and general type $1$-forms on $\mathbb{P}^1$ explicitly. Next, we recall a classical results of Rosenlicht (\cite{Ros74}). He considered the following two equations \begin{equation}\label{Ros-1}
		u'=u^3-u^2,
	\end{equation} \begin{equation}\label{Ros-2}
		u'=\frac{u}{u+1}.
	\end{equation} In view of \cite{NvdPT22}, both the equations (\ref{Ros-1}) and (\ref{Ros-2}) are of general type. Rosenlicht proves that any two distinct nonconstant solutions of these equations (\ref{Ros-1}) and (\ref{Ros-2}) are $C-$algebraically independent, consequently they are new.

	For curves of genus $\geq 1,$ only some theoretical result on the existence of new $1$-forms is known.
	See Lemma 2.18 in \cite{HI03} for the existence of new $1$-forms on elliptic curves. On a curve $X$ of genus $\geq 2,$ it is shown that the set of old forms is contained in a countable union of proper subspaces of $\Omega_X^1$ (\cite[Lemma 2.13]{HI03}). In \cite[Theorem 5.1]{NvdPT22}, Noordman et al. prove that given an effective divisor $\mathrm{D}$ on a curve $X$ of genus $g\geq 2,$  a "generic" element in $$\Omega(\mathrm{D})=\left\lbrace \omega\in \Omega^1_{X} :  \mathrm{div}(\omega)\geq-\mathrm{D}\right\rbrace$$  is a new $1$-form.

	The purpose of this paper is to demonstrate the existence of new and general type meromorphic $1$-forms on curves by constructing them explicitly. In particular, we present systematic constructions of such forms on $\mathbb{P}^1$, elliptic curves, and hyperelliptic curves. Although explicit examples of new and general type $1$-forms on $\mathbb{P}^1$ are known from the results of Rosenlicht and Hrushovski-Itai, our contribution is to provide a general and algorithmic method of construction, which moreover extends to elliptic and hyperelliptic curves of arbitrary genus. This also allows us to address, at least in part, the following algorithmic question for $X=\mathbb{P}^1$: given a pair $(X,\omega)$ (or, equivalently a differential equation), determine its position within the classification of first order differential equations. The main results of the paper are as follows.

	\begin{theorem}\label{main-higher-poles}
		Let $X$ be either $\mathbb{P}^1,$ elliptic or a hyperelliptic curve over $C.$ Given integers $r, n,$ and $m$ such that $r> 0, n\geq 0$\footnote{For elliptic and hyperelliptic curves, $n$ is either zero or an even integer.}\footnote{In case $n=0,$ i.e., $1$-forms with no simple poles, we need $m-r\geq 2-2g.$} and $m\geq 2.$ There exists a collection of new and general type meromorphic $1$-forms on $X$ with $r$ number of zeros, $n$ number of simple poles, and $m$ number of poles of order $\geq 2.$
	\end{theorem}

	\begin{theorem}\label{main-simple-poles}
		Let $X$ be a curve of genus $g$ over $C$ and $\omega$ be a meromorphic $1$-form on $X.$ Assume that $\omega$ has at least $2g+2$ poles with nonzero residues and that no subset with $2g+2$ distinct elements of the set of nonzero residues is linearly dependent over $\mathbb{Q}$\footnote{The existence of such a $1$-forms on an arbitrary curve is follows from \cite[Proposition 1.15, Page 200]{MR95}.}. Then $\omega$ is new and general type\footnote{Theorem \ref{main-simple-poles} extends the result \cite[Lemma 2.23]{HI03} of Hrushovski and Itai to curves of arbitrary genus.}.
	\end{theorem}

	The proof of these theorems effectively uses the behaviour of zeros, poles and residues of meromorphic $1$-forms under pullback and can be thought of as an application of the Riemann-Hurwitz formula. The proof of Theorem \ref{main-higher-poles} is constructive in nature, providing an explicit method to construct new and general type $1$-forms. Furthermore, the locations of the zeros and poles can be chosen arbitrarily, subject to a few mild constraints in the case of elliptic and hyperelliptic curves. Notice that, if a $1$-form $\omega$ satisfy the assumptions of Theorem \ref{main-simple-poles}, then $\omega$ has at least $2g+3$ poles with nonzero residues. Next, we provide an example which implies that the Theorem \ref{main-simple-poles} is sharp. Let $X$ be the smooth hyperelliptic curve given by the affine equation $y^2=(x-e_1)(x-e_2)\cdots(x-e_{2g+1})$ and take
	\begin{equation*}
		\omega= r_1\frac{dx}{x-e_1}+r_2\frac{dx}{x-e_2}+\cdots+r_{2g+1}\frac{dx}{x-e_{2g+1}},
	\end{equation*} where $r_1,\dots,r_{2g+1}\in C.$ Then $\omega$ has $2g+2$ simple poles, namely $(e_1,0),(e_2,0),\dots, (e_{2g+1},0)$ and $\infty.$ Let $r_\infty$ be the residue at $\infty$. Then $r_1+\cdots+r_{2g+1}+r_\infty=0.$ The set of residues $\left\lbrace r_1,\dots,r_{2g+1}\right\rbrace$  can be chosen with no $\mathbb{Q}-$linearly dependent relations among a subset of $\left\lbrace r_1,\dots,r_{2g+1}\right\rbrace$. Note that $\omega$ is defined in the subfield $C(x)$ of $C(X):=C(x,y).$ The pair $(X,\omega)$ is old as it is a pullback of $(\mathbb{P}^1, \omega)$ by the natural degree $2$ projection $\phi:X\rightarrow \mathbb{P}^1.$ Therefore, there exists an old $1$-form having $2g+2$ poles with nonzero residues such that any proper subset of the residues has no $\mathbb{Q}-$linear independent relations.

	As a corollary, we investigate the existence of a new $1$-form in $\Omega(\mathrm{D})=\left\lbrace \omega\in \Omega_X^1:\mathrm{div}(\omega)\geq-\mathrm{D}\right\rbrace $ for an effective divisor $D$ on a curve $X.$ This answers a question proposed in \cite[Section 5]{NvdPT22}. In fact, we consider the existence of a new and general type $1$-form in $\Omega(D).$ In general, there may not exist a new $1$-form in $\Omega(\mathrm{D}).$ For example, let $\mathrm{D}=\left[ 0\right] +\left[ \infty\right] $ on $\mathbb{P}^1.$ If $\omega \in \Omega(\mathrm{D})$ then $\omega=dx/cx$ for some $c\in C\setminus 0.$ The pair $(\mathbb{P}^1, \omega)$  is old as it is a proper pullback of $(\mathbb{P}^1, d(x^n)/cx^n)$ via the morphism $x\mapsto x^n.$ We prove the following result as a corollary of Theorem \ref{main-simple-poles}:
	\begin{corollary}\label{existence-in-omega-D}
		Let $\mathrm{D}=R_1+\cdots+R_n+w_1S_1+\cdots+w_mS_m,$ where $w_i\geq 2,$ be an effective divisor on a curve $X$ of genus $g.$\begin{enumerate}[(i)]
			\item If $n+m \geq 2g+3,$ there exists a family of new and general type $1$-forms in $\Omega(\mathrm{D}).$
			\item If $m\geq1,$ there exists a new $1$-form in  $\Omega(\mathrm{D}).$
		\end{enumerate}
	\end{corollary}

	We also derive a lower bound of the size of new and general type $1$-forms in $\Omega(\mathrm{D})$ for some restricted class of $\mathrm{D}$ (Theorem \ref{dim-thm} ). Regarding this, we construct new and general type $1$-forms with a given polar divisor. The result for $\mathbb{P}^1$ reads as follows. If $\mathrm{D}=R_1+\dots+z_0\infty+R_n+w_1S_1+\dots+w_mS_m$ with $n\geq 1, m\geq 2, z_0\geq 1,$ and $w_i$ sufficiently large. Then there exists a vector space $W\subseteq \Omega(\mathrm{D})$ of dimension at least $n \lfloor \frac{m}{2} \rfloor$ such that any nonzero element of $W$ is new and general type. Here, $\lfloor r \rfloor$ denotes the greatest integer less than or equal to $r$. Here $\lfloor r \rfloor$  means the greatest integer less than or equal to $r.$

	We present an algorithm that decides whether a $1$-form on $\mathbb{P}^1$ is new or old. Given a pair $(\mathbb{P}^1, \omega),$ if it is a proper pullback of another pair $(\mathbb{P}^1, \eta)$ via a morphism $\phi,$  then we show that the possible branch data of $\phi$ can be determined by $\omega.$ Thus, if one can realize certain branch data  by a morphism $\phi:\mathbb{P}^1\rightarrow \mathbb{P}^1,$ then  the type of the $1$-form $\omega$ can be decided. The realization of branch data by a morphism  is the classical \emph{Hurwitz realization problem} for the Riemann sphere. It is still an open problem.  But for some particular classes of branch data, the answer to the realization problem is known (see Section \ref{hur-prob}).  As a result of that, our algorithm is always applicable to the equations of the form \begin{equation}\label{diff-equ-laurent-poly}
		u'=g(u) \text{ }\text{where}\text{ }g\in \mathbb{C}[Z,Z^{-1}]\text{ }\text{or,}\text{ } \frac{1}{g}\in \mathbb{C}[Z,Z^{-1}].
	\end{equation}
	Consequently, the place of differential equations (\ref{diff-equ-laurent-poly}) can be decided in the classification. Finally, we establish explicit criteria for equations to be new and of general type, and we verify that the equations (\ref{Ros-1}) and (\ref{Ros-2}) satisfy these criteria, thereby providing a new proof that they are of general type and new.

	\textbf{Organization of the Paper.} In Section 2, we briefly review the Hurwitz realization problem, its solutions, and  the existence of meromorphic $1$-forms on curves. Section 3 is about the pullback of pairs and  general type $1$-forms. In Section 4, the main results are proven.  The algorithm and its applications are displayed in Section 5, with some explicit criteria and examples of new and general type equations at the end.
	
	\textbf{Acknowledgements.}
	The author would like to thank Prof. Varadharaj Ravi Srinivasan and Prof. Chetan Balwe for the many helpful discussions, suggestions, and references they provided during the preparation of this article. The author is grateful to IISER Mohali for providing a PhD fellowship.

	\section{Hurwitz realization problem and existence of meromorphic $1$-forms on curves}
	\subsection{Hurwitz realization problem of branch covering for the Riemann sphere.}\label{hur-prob}
	Let $\phi:\mathbb{S}^2\rightarrow \mathbb{S}^2$ be a $d$-fold branch covering, or equivalently, a meromorphic function on the Riemann sphere with $Q_1,\dots, Q_n \in\mathbb{S}^2$ be its branch points. Then for every $i=1,\dots,n,$ there exist $k_i$ points $P_{i,1},\dots,P_{i,k_i}\in \mathbb{S}^2$ and positive integers $d_{i,1},\dots,d_{i,k_i}$ such that \begin{itemize}
		\item $d_{i,j}> 1$ for some $j\in\left\lbrace 1,\dots,k_i\right\rbrace,$
		\item $\sum_{j=1}^{k_i}d_{i,j}=d,$ and
		\item for every $j=1,\dots,k_i$ the map $\phi$ on some neighbourhood of $P_{i,j},$ with $\phi(P_{i,j})=Q_i,$ looks like $z\longmapsto z^{d_{i,j}}.$
	\end{itemize}
	Thus, a branch covering  $\phi$ defines for each of its branch points $Q_i$ a set $\Phi_i=\left\lbrace d_{i,1},\dots, d_{i,k_i}\right\rbrace$ as a partition of $d.$ The collection $\Phi= \left\lbrace \Phi_1,\dots \Phi_n \right\rbrace$ is called the \emph{branch data} of $\phi,$ denoted by $\mathcal{D}(\phi).$ In this language, the Riemann-Hurwitz formula implies
	\begin{equation}\label{Rie-Hur}
		\sum_{i=1}^{n}k_i=(n-2)d+2.
	\end{equation}
	
	\begin{definition} Let $d, n, k_1,\dots, k_n$ be positive integers. An \emph{abstract branch data} $\mathcal{D}:=\mathcal{D}(d;n;k_1,\dots,k_n)$ of degree $d>1$  is a collection  $
		\Phi= \left\lbrace \Phi_1,\dots \Phi_n \right\rbrace$ where each $
		\Phi_i=\left\lbrace d_{i,1},\dots, d_{i,k_i}\right\rbrace $ is a set of positive integers such that $d_{i,j}> 1$ for some $j,\  \sum_{j=1}^{k_i}d_{i,j}=d$ and satisfy the equation (\ref{Rie-Hur}).
	\end{definition}
	A natural existence problem is the following: for which abstract branch data $\mathcal{D},$ does there exist a meromorphic function  on the Riemann sphere such that $\mathcal{D}(\phi)=\mathcal{D}?$ In case it exists, call the branch data $\mathcal{D}$ \emph{realizable}.

	The existence problem for a meromorphic function is a sub-case of the existence problem for branched covering $\phi:M_1\rightarrow M_2$ between closed Riemann surfaces, which is posed by Hurwitz (\cite{Hur91}). Many authors studied this problem; for a survey, see \cite{KB01}. In general, the answer to the above existence problem is negative, i.e., there are non-realizable abstract branch data. For example, the branch data $\left\lbrace \left\lbrace 3,1\right\rbrace ,\left\lbrace 2,2\right\rbrace, \left\lbrace 2,2\right\rbrace \right\rbrace $ is not realizable. There are some important particular cases when the answer to the existence problem is  known:\begin{enumerate}[(i)]
		\item (Polynomial branch data) A branch data $\mathcal{D}=\mathcal{D}(d;n;k_1,\dots,k_n)$ such that there is an $i, 1\leq i\leq n$ for which $\#\Pi_i =1,$ equivalently $k_i=1$ and $d_{i,1}=d$ is realizable. A meromorphic function $\phi$ is polynomial  if and only if there exists $i$ such that $k_i=1$ and $d_{i,1}=d$ in $\mathcal{D}(\phi).$ Hence, the realizability problem is equivalent to the existence of polynomial function. This was proved much earlier in \cite{RT65}.
		\item (Laurent polynomial branch data) A branch data $\mathcal{D}=\mathcal{D}(d;n;k_1,\dots,k_n)$ such that there is an $i, 1\leq i \leq n$ for which $\#\Pi_i=2,$ equivalently $k_i=2$ and $d_{i,1}=k, d_{i,2}=d-k$ for some $k, 1\leq k\leq d,$ is always realizable if $n\geq 4;$ for $n=3$ the complete list of non-realizable data is known. The realizability problem is equivalent to the existence of Laurent polynomials. For reference, see  \cite{PF09}.  
		
		\item All the branch data of degree $d=2,3,5,7$ are realizable, and for $d=4$ there is only one non-realizable data, namely $\left\lbrace \left\lbrace 3,1\right\rbrace ,\left\lbrace 2,2\right\rbrace, \left\lbrace 2,2\right\rbrace \right\rbrace$ (See \cite[Table 2]{Zhe06}). 
	\end{enumerate}

	\subsection{Existence of meromorphic $1$-forms on curves.}\label{exist-1-form} In the rest of the paper, $E$ will mean an elliptic curve defined by the affine equation $y^2=(x-e_1)(x-e_2)(x-e_3);$ $H$ will denote a hyperelliptic curve with affine equation $y^2=(x-e_1)\cdots(x-e_{2g+1}).$ The point at infinity will be denoted by $\infty.$ Here, we record some known results on the existence of a certain class of meromorphic $1$-forms on curves.
	
	Given a divisor $\mathrm{D}$ of degree $-2$ on $\mathbb{P}^1,$ it is easy to construct a $1$-form $\omega$ with $\mathrm{div}(\omega)=\mathrm{D}.$ This existence also follows from \cite[Problems VI.3]{MR95}: Let $\mathrm{D}$ be a divisor on an algebraic curve $X$ of genus $g$ such that $\mathrm{deg}(\mathrm{D})=2g-2$ and $\mathrm{dim}\  \mathcal{L}(\mathrm{D})=g.$ Then $\mathrm{D}$ is a canonical divisor. Therefore, for $X=\mathbb{P}^1,$ any divisor of degree $-2$ is a canonical divisor. For arbitrary curves, there are some classes of $1$-forms whose existence is well known. The next proposition implies the existence of $1$-forms on a curve $X$ satisfying the assumption of Theorem \ref{main-simple-poles}.

	\begin{proposition}\cite[Proposition 1.15, Page 200]{MR95}
		Given an algebraic curve $X,$ a finite set of points $\left\lbrace P_1,\dots, P_n\right\rbrace $ on $X$ and a corresponding subset $\left\lbrace c_1,\dots,c_n\right\rbrace $ of $C,$ there is a meromorphic $1$-form $\omega$ on $X$ whose only poles are simple poles at $P_1,\dots,P_n$ with $\mathrm{Res_{P_i}(\omega)}=c_i$ for each $i,$ if and only if  $\ \sum_{i}^{n}c_i=0.$ 
	\end{proposition}
	Next, we shall describe uniformizer parameters on elliptic and hyperelliptic curves and use them to construct a class of meromorphic $1$-forms.

	(1) Elliptic curve: A point $P=(x,y)$ on the elliptic curve $E$ is called special if $P=\mathfrak{I}(P)$ where $\mathfrak{I}(P)$ is the inverse of $P$ in the elliptic group law. Otherwise, $P$ is called ordinary. The elliptic curve $E$ has four special points, namely $T_1=(e_1,0), T_2=(e_2,0), T_3=(e_3,0)$ and $\infty.$ Let $t_P$ be the uniformizer at $P.$ For an ordinary point $P=(x_{P},y_{P}), t_P=x-x_{P}, t_{T_i}=y$ for all $i$ and $t_{\infty}=x/y.$

	On $E,$ define the following class of $1$-forms:  \begin{equation}\label{1-form-elliptic}
		\omega=\frac{(x-x_{P_1})^{u_1}\cdots(x-x_{P_r})^{u_r}}{(x-x_{R_1})\cdots(x-x_{R_n})(x-x_{S_1})^{w_1}\cdots(x-x_{S_m})^{w_m}}(x-e_1)^{k_1}(x-e_2)^{k_2}(x-e_3)^{k_3}y^ldx\end{equation} where $u_i\geq 1, w_i\geq2$ positive integers, $k_1, k_2, k_3, l$ are integers, and $P_i=(x_{P_i}, y_{P_i}), R_j=(x_{R_j},y_{R_j}), S_k=(x_{S_k}, y_{S_k})$ are distinct ordinary points of $E.$ 
	Then \begin{align*}
		\mathrm{div}(\omega)= &\sum_{i=1}^{r}u_i( P_i + \mathfrak{I}(P_i) )-\sum_{i=i}^{n}( R_i + \mathfrak{I}(R_i) )-\sum_{i=1}^{m}w_i( S_i + \mathfrak{I}(S_i)) \\ & +\sum_{i=1}^{3}(2k_i+l+1) T_i+(2n+\sum_{i=1}^{m}2w_i-\sum_{i=1}^{n}2u_i-\sum_{i=1}^{3}2k_i-3l-3) \infty.
	\end{align*}

	(2) Uniformizer of hyperelliptic curve: An \emph{hyperelliptic involution}  is a morphism $\mathfrak{I}:H\rightarrow H$  defined by $\mathfrak{I}(x,y)=(x,-y)$ for $(x,y)\in H.$ A point $P\in H$ is called special if $\mathfrak{I}(P)=P.$ Otherwise, $P$ is called ordinary. The curve $H$ has $2g+2$ special points, namely $T_i=(e_i, 0), i=1,\dots, 2g+1,$ and $\infty.$ Let $t_P$ be the uniformizer at $P.$ For an ordinary point  $P=(x_{P},y_{P}), t_P=x-x_P, t_{T_i}=y$ for all $i$ and $t_{\infty}=x^g/y.$

	On $H,$ define the following class of $1$-forms: \begin{equation}\label{1-form-H}
		\omega=\frac{(x-x_{P_1})^{u_1}\cdots(x-x_{P_r})^{u_r}}{(x-x_{R_1})\cdots(x-x_{R_n})(x-x_{S_1})^{w_1}\cdots(x-x_{S_m})^{w_m}}(x-e_1)^{k_1}\cdots(x-e_{2g+1})^{k_{2g+1}}y^ldx
	\end{equation}
	where $u_i\geq 1, w_i\geq 2$ positive integers, $k_1, k_2, k_3, l$ are integers, and $P_i=(x_{P_i}, y_{P_i}), R_j=(x_{R_j},y_{R_j}), S_k=(x_{S_k}, y_{S_k})$ are distinct points of $H.$ Then \begin{align*}
		\mathrm{div}(\omega)= &\sum_{i=1}^{r}u_i( P_i + \mathfrak{I}(P_i)  )-\sum_{i=i}^{n}( R_i + \mathfrak{I}(R_i) )-\sum_{i=1}^{m}w_i( S_i + \mathfrak{I}(S_i)) \\ &+\sum_{i=1}^{2g+1}(2k_i+l+1) T_i  +(2n+\sum_{i=1}^{m}2w_i-\sum_{i=1}^{n}2u_i-\sum_{i=1}^{2g+1}2k_i-l(2g+1)-3) \infty.
	\end{align*}
	
	In the class of $1$-forms defined above, if $l$ is either zero or an even integer, then $\omega$ is defined in the subfield $C(x)$ of $ C(x,y).$ Hence $\omega$ is an old form. If $l$ is an odd integer, the order of $\omega$ at the special points is either zero or an even integer. Then it follows that given numbers $r> 0, n\geq0, m\geq 2$ such that $n$ is either zero or even number; one can construct a $1$-form on an elliptic or hyperelliptic curve having support at special points with $r$ number of zeros, $n$ number of simple poles, and $m$ number of poles of order $\geq 2.$ Indeed if $r=2p+1$ odd, one can use a special point and $2p$ number of ordinary points.

	\section{ Pullback of pairs and general type $1$-forms}\label{gen-type-equ} In this section, we note some observations on the behaviour of zeros, poles, and residues of $1$-forms under pullback. A pair $(X,\omega)$ is called \emph{pullback} of another pair $(Y,\eta)$ if there is a nonconstant morphism $\phi:X\rightarrow Y$ such that $\phi^*\eta=\omega.$ Denote it by $(X, \omega)\overset{\phi}{\longrightarrow}(Y, \eta).$ Ramification index at a point $P$ will be denoted by $e_{P}.$ 
	\begin{lemma}
		Let $(X, \omega)\overset{\phi}{\longrightarrow}(Y, \eta)$ with $\phi(P)=Q.$ Then $e_{P}(\mathrm{ord}_{Q}(\eta)+1)=\mathrm{ord}_{P}(\omega)+1.$ 
		
		\begin{proof} Let $m=\mathrm{ord}_{Q}(\eta).$ Write $\eta= ht^mdt,$ where $t$ is a uniformization parameter at $Q$ with $h(Q)\neq 0.$ We have $\phi^*t= us^{e_P},$ where $s$ is a uniformization parameter at $P$ with $u(P)\neq 0, \infty.$ Then $$\omega=\phi^*\eta = (\phi^*h)(\phi^*t^m)d(us^{e_P})= (\phi^*h)(u^ms^{e_Pm})[e_Pus^{{e_P}-1}+(du/ds)s^{e_P}]ds,$$ implies \begin{equation}\label{ram-ord}
				\mathrm{ord}_{P}(\omega)+1=e_{P}(\mathrm{ord}_{Q}(\eta)+1).
			\end{equation}
		\end{proof}
	\end{lemma}
	
	\begin{lemma}\label{obv}
		Let $(X, \omega)\overset{\phi}{\longrightarrow}(Y, \eta)$ with $\mathrm{div}(\omega)= \sum_{i=1}^{u_\omega}n_iP_{\omega i}-\sum_{j=1}^{v_\omega}R_{\omega j}-\sum_{k=1}^{w_\omega} m_kS_{\omega k}$ and $\mathrm{div}(\eta)= \sum_{i=1}^{u_\eta}n_iP_{\eta i}-\sum_{j=1}^{v_\eta}R_{\eta j}-\sum_{k=1}^{w_\eta} m_kS_{\eta k}$ where all $m_k\geq 2$ for all $k.$ Then the following holds: 
		\begin{enumerate}[(i)]
			\item\label{obv-simple-pole} for all $j,$  $\phi(R_{\omega j})\in \left\lbrace R_{\eta 1},\dots, R_{\eta v_\eta}\right\rbrace $ and $\phi^{-1}(R_{\eta j})\subseteq \left\lbrace R_{\omega 1},\dots, R_{\omega v_\omega}\right\rbrace ;$ 
			\item \label{obv-ram-ind-&-highere-pole} for all $k,$ $e_{S_{\omega k}}$ divides $\left| \mathrm{ord}_{S_{\omega k}}(\omega)+1\right|,  \phi(S_{\omega k})\in \left\lbrace S_{\eta 1},\dots,  S_{\eta w_\eta}\right\rbrace $ and $\phi^{-1}(S_{\eta k})\subseteq \left\lbrace S_{\omega 1},\dots, S_{\omega w_\omega}\right\rbrace ;$ 
			\item\label{obv-zero}
			for all $i, e_{P_{\omega i}}$ divides $\mathrm{ord}_{P_{\omega i}}(\omega)+1$ and $\phi^{-1}(P_{\eta i})\subseteq \left\lbrace P_{\omega 1},\dots,P_{\omega u_\omega} \right\rbrace.$ Also for any $i, \phi(P_{\omega i})$ is not a zero of $\eta$ if and only if $e_{P_{\omega i}}=\mathrm{ord}_{P_{\omega i}}(\omega)+1;$   
			\item\label{obv-fiber} if $ S_{\omega k_1}, S_{\omega k_2} \in \phi^{-1}(S_{\eta j})$ for some $ k_1, k_2$ then 
			\begin{equation*}\label{equ-ram-fiber}
				\frac{\mathrm{ord}_{S_{\omega k_1}}(\omega)+1}{e_{S_{\omega k_1}}}=\frac{\mathrm{ord}_{S_{\omega k_2}}(\omega)+1}{e_{S_{\omega k_2}}}.
			\end{equation*}The same is true for zeros for $\omega$ and $\eta$ also.
			\item\label{obv-residue} If $Q$ is a pole of $\eta$ with residue $a$, then each $P\in \phi^{-1}(Q)$ is a pole of $\omega$ with residue $e_{P} a.$
		\end{enumerate}
		\begin{proof}
			The proof follows from equation (\ref{ram-ord}).
		\end{proof}
	\end{lemma}

	\begin{proposition}\label{key-proposition}\begin{enumerate}[(i)]
			\item If $(X,\omega)$ is of exact type, then $\omega$ has a pole of order $\geq 2,$ and all poles of $\omega$ are of order $\geq 2$ with zero residues.
			\item If $(X,\omega)$ is of exponential type, then $\omega$ has a simple pole, and all poles of $\omega$ are simple with  residues $cm_i,$ for some unique $c\in C$ and $m_i\in\mathbb{Z}$ vary with simple poles.   
			\item If $(X,\omega)$ is of Weierstrass type, then $\omega$ is holomorphic.
			\item If $\omega\in \Omega_X^1$ has a pole of order $\geq 2$ with nonzero residue, then $(X,\omega)$  is of general type.
			\item If $\omega\in \Omega_X^1$ has a pole of order $\geq 2$ and a simple pole, then $(X,\omega)$  is of  general type.
		\end{enumerate}
		\begin{proof}
			By definition, a pair $(X,\omega)$ is of exact type (respectively of exponential type, respectively of  Weierstrass type) if and only if $(X,\omega)$ is pullback of the pair $(\mathbb{P}^1, dx)$ (respectively $(\mathbb{P}^1, \frac{dx}{cx}),$ respectively $(E,\frac{dx}{y})$). The $1$-form $dx$ has only one pole of order $\geq 2$ with zero residues and the $1$-form $\frac{dx}{cx}$ has only  two simple poles with nonzero residues. The $1$-form $\frac{dx}{y}$ on $E$ is holomorphic. Now the proof of the proposition follows from Lemma \ref{obv}.  
		\end{proof}
	\end{proposition}

	\begin{corollary}\label{rat-aut}(\cite[Remark 7.5]{KRS24})\begin{enumerate}[(i)]
			\item $(\mathbb{P}^1, \frac{dx}{g(x)})$ is of exact type if and only if the partial fraction expression of $\frac{1}{g(x)}$ is of the form
			\begin{equation*}
				h(x)+\sum_{i=1}^{n}\sum_{j=2}^{n_i}\frac{d_{ij}}{(x-c_i)^j}
			\end{equation*} where $d_{ij}, c_i\in C$ and $h(u)$ is a polynomial over $C.$ Furthermore, $(\mathbb{P}^1, \frac{dx}{g(x)})$ is of exact type and new if and only if $\frac{1}{g(x)}=\frac{d}{(x-c)^2}$ for some $c, d\in C.$  
			
			\item $(\mathbb{P}^1, \frac{dx}{g(x)})$  is of exponential type if and only if the partial fraction of $\frac{1}{g(x)}$ is of the form 
			\begin{equation*}
				c\sum_{i=1}^{n}\frac{m_i}{(x-c_i)}
			\end{equation*}where $c_i,c\in C$ and $m_i$ are nonzero integers. 
		\end{enumerate}
		\begin{proof}
			The proof follows from Proposition \ref{key-proposition}.
		\end{proof}
	\end{corollary}

	\begin{proposition}\label{gen-type-prop}(\cite[Proposition  3.3]{NvdPT22})\begin{enumerate}[(i)]
			\item If $(X, \omega)\overset{\phi}{\longrightarrow}(Y, \eta)$ and $(X,\omega)$ is of general type then $(Y,\eta)$ is also of general type.
			\item If $(X,\omega)$ is of general type, then $\omega$ has at least one zero. 
		\end{enumerate}  
	\end{proposition}
	\begin{proposition}\label{deg-bound}
		Let $(X, \omega)\overset{\phi}{\longrightarrow}(Y, \eta)$ and $\omega$ be general type. Then $\mathrm{deg}(\phi)\leq \frac{1}{2}(\mathrm{deg}(\mathrm{div}_0(\omega))+m) ,$ where $\mathrm{div}_0(\omega)$ denotes the divisor of zeros of $\omega$ and $m$ is the number of zeros of $\omega.$
		
		\begin{proof}
			As $\omega$ is general type so is $\eta.$ Let $Q$ be a zero of $\eta$ and $\phi^{-1}(Q)=\left\lbrace P_1,P_2,\dots,P_n\right\rbrace.$ By Lemma \ref{obv}, for each $i, 1\leq i\leq n,$ $e_{P_{i}}$ is a proper divisor of $\mathrm{ord}_{P_{i}}(\omega)+1,$ implies $e_{P_i}\leq \frac{1}{2}(\mathrm{ord}_{P_{i}}(\omega)+1).$ 	Hence\begin{equation*}
				\mathrm{deg}(\phi)=\sum_{P_{i}\in \phi^{-1}(Q)}e_{P_{i}}\leq \sum_{i=1}^{n}\frac{1}{2}(\mathrm{ord}_{P_{i}}(\omega)+1)\leq \frac{1}{2}(\mathrm{deg}(\mathrm{div}_0(\omega))+m).
			\end{equation*} 
		\end{proof} 
	\end{proposition}

	\section{ Existence of new and general type $1$-forms}
	
	\subsection{Proof of main theorems and related examples} 
	Let $X$ and $Y$ be two curves with genus $g_X$ and $g_Y$ respectively. Let $\phi:X\rightarrow Y$ be a morphism with $d:=\mathrm{deg}(\phi).$ By Riemann-Hurwitz formula, \begin{equation}\label{Rie-Hur-Formula}
		2g_X-2=d(2g_Y-2)+ \sum_{P\in X}(e_P-1).	\end{equation} If $g_X, g_Y\geq 2,$ then $e_P\leq d\leq\frac{2g_X-2}{2g_Y-2}.$  If $g_{Y}=1,$ then $2g_{X}-2=\sum_{P\in X}(e_{P}-1)$ implies $e_p\leq 2g_X-1.$ Therefore, if  $(X, \omega)\overset{\phi}{\longrightarrow}(Y, \eta),$ with $g_Y\geq 1,$ and $P$ is either a zero or pole of $\omega,$ then  \begin{equation}\label{bound-ramification-index}
		e_{P}\leq 2g_X-1.
	\end{equation}
	
	For $n\in \mathbb{N},$ let $\mathfrak{D}(n):=\left\lbrace m\in \mathbb{N} : m \text{ does not have divisor less than equal to n except 1}\right\rbrace.$ Note that  $n_1\leq n_2$ implies $\mathfrak{D}(n_2)\subseteq \mathfrak{D}(n_1).$

	\begin{lemma}\label{main-lemma}
		Let $\omega$ be a meromorphic $1$-form on a curve $X$ of genus $g_X$ with \begin{equation*}
			\mathrm{div}(\omega)= \sum_{i=1}^{r}u_iP_i-\sum_{i=1}^{n}R_i-\sum_{i=1}^{m}w_iS_i, \text{ }\text{$w_i\geq 2.$}
		\end{equation*}\begin{enumerate}[(i)]
			\item Let $r>0, n>0$ and $m\geq 2.$ If  $w_i-1\in \mathfrak{D}(2g_X+r+n-1)$ for $1\leq i\leq m,$ and there is an $i$ such that $ w_i\neq w_j$ for $1\leq j\neq i\leq m,$ then $\omega$ is new and general type.
			\item Let $r>0, n=0$ and $m\geq 2.$ If $w_i-1\in \mathfrak{D}(2g_X+r+m),1\leq i\leq m,$ are prime, there is an $i$ such that $w_i\neq w_j$ for $ 1\leq j\neq i\leq m$ and some $u_i$ is greater than equal to $(\sum_{i=1}^{m}w_i)-m,$ then $\omega$ is new and general type. 
		\end{enumerate}
		\begin{proof} (i) By Proposition \ref{key-proposition}, $(X,\omega)$ is of general type . If $\omega$ is old, then there is a morphism $\phi$ and a pair $(Y,\eta)$ such that $(X, \omega)\overset{\phi}{\longrightarrow}(Y, \eta)$ with $\mathrm{deg}(\phi):=d\geq 2.$ The pair $(Y,\eta)$ is of general type; hence $\eta$ has a zero, say at $Q$ (Proposition \ref{gen-type-prop}). Let $R=R_i$ for some $i.$ Then every element in $\phi^{-1}(Q)$ is a zero of $\omega$ and every element in $\phi^{-1}(\phi(R))$ is a simple pole of $\omega.$
			
			We claim that $e_{S_i}=1$ for all $i.$ If $g_Y\geq 1$ then by equation (\ref{bound-ramification-index}), $e_{S_i}\leq 2g_X-1$$\leq 2g_X+r+n-1.$
			Next, let $g_Y=0.$ By Riemann-Hurwitz formula 
			\begin{equation*}
				\sum_{P\in X}(e_P-1)= 2g_X+2d-2. 
			\end{equation*}Then for any pole $S_i$ of $\omega,$
			\begin{equation*}
				(e_{S_i}-1)+ \sum_{P\in \phi^{-1}(Q)}(e_{P}-1)+\sum_{T\in \phi^{-1}(\phi(R))}(e_{T}-1) \leq 2g_X+2d-2.
			\end{equation*} This implies,
			$e_{S_i}\leq (2g_X+2d-2)-(d-\#\phi^{-1}(Q))-(d-\#\phi^{-1}(\phi(R_1)))+1\leq 2g_X+r+n-1.$ Also by Lemma \ref{obv} (ii), $e_{S_i}$ divides $\left| \mathrm{ord}_{S_i}(\omega)+1\right| = w_i-1.$ Since $w_i-1\in \mathfrak{D}(2g_X+r+n-1), w_i-1$ does not have divisor less then equal to $2g_X+r+n-1.$ Hence   $e_{S_i}=1$ for all $i.$ This proves the claim.

			Now $d\geq 2, e_{S_i}=1$ for all $i$ and $\phi^{-1}(\phi(S_i))\subseteq \left\lbrace S_1,\dots,S_m\right\rbrace,$ implies $\phi^{-1}(\phi(S_i))$ must contains $S_j$ for some $j,$ $1\leq j\neq i\leq m.$ Then by Lemma \ref{obv} (iv) $w_i=w_j,$ a contradiction. Therefore, no such morphism $\phi$ and a pair  $(Y,\eta)$ exists, and $\omega$ is new.

			(ii) We shall show that there is no morphism $\phi$ and a $1$-form $\eta$ on a curve $Y$ such that $(X, \omega)\overset{\phi}{\longrightarrow}(Y, \eta)$ with $\mathrm{deg}(\phi):=d\geq 2.$ Let us assume the contrary. If $g_Y\geq 1,$ then using the same argument as in the proof of (i), we get $e_{S_i}\leq 2g_X-1$ for all $i,$ and the assumptions on $w_i$ lead to a contradiction. 
			
			Let $g_Y=0.$ We divide the proof into two parts. First, we show that there exists no pair $(\mathbb{P}^1,\eta)$ such that $\eta$ has at least one zero and at least two poles of order $\geq 2.$  As $\eta$ has two poles of order $\geq 2,$ by Lemma \ref{obv} (ii), for each $i$ there is a $j, j\neq i, 1\leq i,j\leq m$ such that $S_j\notin \phi^{-1}(\phi(S_i)).$ Otherwise, for any $i,$ $\phi^{-1}(\phi(S_i))=\left\lbrace S_1,\dots, S_m\right\rbrace,$ which contradicts that $\eta$ has two poles of order $\geq2.$ Let $Q$ be a zero of $\eta.$ By Riemann-Hurwitz formula, for any $S_i$\begin{equation*}
				(e_{S_i}-1)+\sum_{P\in \phi^{-1}(Q)}(e_{P_i}-1)+\sum_{S_i\in\phi^{-1}(\phi(S_j))}\leq 2g_X+2d-2.
			\end{equation*}Hence $e_{S_i}\leq 2g_X+r+m$ for all $i.$ Then again, using the same argument as in (i), because of the assumptions on $w_k,$ we have $w_i=w_j$ for some $j\neq i,$ a  contradiction.
			
			To complete the proof, we need to show that there is no morphism $\phi$ and $1$-form $\eta$ with $(X, \omega)\overset{\phi}{\longrightarrow}(\mathbb{P}^1, \eta)$  with $\mathrm{deg}(\phi):=d\geq 2$ where either $\eta$ has no zero (hence $\eta$ has only a pole of order $2$) or $\eta$ has one zero and only one pole of order $> 2.$  In the first case, by Lemma \ref{obv} (iii), $e_{P_i}= u_i+1.$  Also all poles of $\omega$ will map to the only pole, say $Q_1,$ of order $2$ of $\eta.$ Then $e_{S_i}= w_i-1$ for $1\leq i\leq m$ and $d= \sum_{P\in \phi^{-1}(Q_1)}e_{P}= \sum_{i=1}^{m}w_i-m.$ But by assumption $e_{P_i}=u_i+1 > d$, a contradiction. This in fact proves that $\omega$ is not exact type. Also, $\omega$ cannot be exponential type, or Weierstrass type, follows from Proposition \ref{key-proposition}.

			Let $Q_2$ be the only pole of order $\tilde{w}> 2$ of $\eta$ in the second case. Then all the poles of $\omega$ will map to $Q_2.$ As $w_i-1$ are prime, $e_{S_i}$ is either $1$ or $w_i-1.$ But $\tilde{w}>2$ implies all $e_{S_k}=1.$ Again by the same argument,  $w_i=w_j$ for some $j\neq i,$ a contradiction. Therefore, $(X, \omega)$ does not have a proper pullback, and $(X,\omega)$ is new and of general type.     
		\end{proof}
	\end{lemma}
	
	\begin{remark} 
		In Lemma \ref{main-lemma}, (i) can be replace by (i)$^*:$ Let $r>0, n>0$ and $m\geq 2.$ If $u_i+1\in \mathfrak{D}(2g_X+r+n-1)$ and  $u_i\neq u_j$ for $1\leq i\neq j\neq i\leq m,$  then $\omega$ is new and general type.
	\end{remark}

	\begin{proof}[Proof of Theorem \ref{main-higher-poles}] The proof follow from Lemma \ref{main-lemma}, once we prove that there exists  $1$-forms on $\mathbb{P}^1,$ elliptic and hyperelliptic curves satisfying conditions of Lemma \ref{main-lemma}. \begin{enumerate}[(i)]
			\item Let $r>0, n>0$ and $m\geq2.$ For $\mathbb{P}^1,$ the existence is obvious. For elliptic and hyperelliptic curves, if $n$ is even, one can construct $1$-forms satisfying Lemma \ref{main-lemma} (i) using  (\ref{1-form-elliptic}) and (\ref{1-form-H}), as explained in Section \ref{exist-1-form}.
			\item Let $r>0, n=0, m\geq 2.$ In this case, $1$-forms satisfying conditions of Lemma \ref{main-lemma}(ii) exists if $m-r\geq 2-2g.$ The last condition ensures that after choosing one of the $u_i$ big enough, one can still choose other zeros of order at least one.             
		\end{enumerate}     
	\end{proof}

	Next, we use Lemma \ref{main-lemma} to produce some easy examples of new and general type $1$-forms on $\mathbb{P}^1,$ elliptic and hyperelliptic curves.

	\begin{examples} Let $\mathbb{P}^1$ be projective line over $\mathbb{C}.$
		\begin{enumerate}[(i)] 
			\item Let $r=5, n=2$ and $m=4.$ Then $2g_X+r+n-1=6.$ Choose $w_1=14, w_2=w_3=18$ and $w_4=20.$ They satisfy the conditions of Lemma \ref{main-lemma}. To get a new and general type $1$-form, we need to choose $u_1,\dots,u_5$ such that $u_1+\dots+u_5-2-14-18-18-20=-2,$ i.e., $u_1+\dots+u_5=70.$ There are many choices and for each choice we get a $1$-form. For example, Let  $u_1=10, u_2=24, u_3=11, u_4=21, u_5=4,$ then the $1$-form is
			$$\frac{(x-1)^{10}(x-2)^{24}(x-3)^{11}(x-4)^{21}(x-5)^4dx}{(x-6)(x-7)(x-8)^{14}(x-12)^{18}(x-23)^{18}(x-25)^{20}}.$$
			\item Let $r=3, n=0,$ and $m=7.$ Then $2g_X+r+m=10$ and $m-r=4\geq2.$ Choose $w_1=12, w_2= w_3=w_4=14, w_5=w_6= 18, w_7=20$ and $u_1=103.$ The $1$-form is \begin{equation*}
				\omega=\frac{x^{103}(x-2)^2(x-3)^3dx}{(x-1)^{12}(x-5)^{14}(x-6)^{14}(x-7)^{14}(x-8)^{18}(x-9)^{18}(x-10)^{20}}
			\end{equation*} is new and general type by Lemma \ref{main-lemma}.
		\end{enumerate}
		
	\end{examples}

	\begin{examples} Let $E$ be the elliptic curve be given by the equation $y^2=x^3-4x$ over $\mathbb{C}.$
		\begin{enumerate}[(i)] 
			\item Let $r=3, n=2, m=2.$ Then  $2g_X+r+n-1=6.$ Choose $w_1=12, w_2=14$ and consider the $1$-form on $E$
			$$\omega=\frac{x(x-3)^{12}}{(x+1)(x-2)^{7}(x+2)^{8}}ydx.$$ Then $\mathrm{div}(\omega)=12 \left[ (3,\sqrt{15})\right] + 12\left[ (3,-\sqrt{15})\right] +4 \left[  (0,0)\right] -\left[ (-1,\sqrt{3})\right] -\left[ (-1,-\sqrt{3})\right] -12\left[ (2,0)\right] -14\left[  (-2,0)\right] .$ By Lemma \ref{main-lemma}, $\omega$ is new and general type.
			\item Let $r=3, n=0$ and $m=4.$ Then $2g_X+r+m=9$ and $m-r=1\geq 0.$ Let $w_1=12, w_2=12, w_3=18, w_4=24$ and $u_1=62.$ Consider the $1$-form on $E$  $$\omega=\frac{(x-3)^{2}x^{30}}{(x+2)^{10}(x-2)^{13}(x-5)^{12}}ydx.$$ Then $\mathrm{div}(\omega)= 2\left[ (3,\sqrt{15})\right]+ 2\left[ (3,-\sqrt{15})\right]+62\left[ (0,0)\right] -18\left[ (2,0)\right]- 24\left[ (-2,0)\right]-12\left[ (5, \sqrt{105})\right] -12 \left[ (5, \sqrt{105})\right] $ and by Lemma \ref{main-lemma}, $\omega$ is new and general type.
		\end{enumerate}
	\end{examples} 
	
	\begin{examples} Let $H$ be the hyperelliptic curve of genus $2$ given by $y^2=x(x+1)(x+2)(x-2)(x-3)$ over $\mathbb{C}.$
		\begin{enumerate}[(i)]
			\item Let $r=5, n=2, m=5.$ Then $2g_X+r+n-1=10.$ The following $1$-form $\omega$ on $H$ defined by
			$$\omega=\frac{(x-2)^3(x+4)^{36}(x+5)^{36}}{x^{13}(x-1)(x-4)^{30}(x-5)^{32}}ydx.$$ 
			The pair $(H,\omega)$ is new and of general type.  
			\item Let $r=3, n=0, m=2.$ Then $2g_X+r+m=9$ and $m-r=-1\geq -2.$ Consider the $1$-form on $H$ defined by
			$$\omega=\frac{x^{14}(x-1)^2}{(x+2)^{8}(x-2)^{10}}ydx$$ with $\mathrm{div}(\omega)=  30\left[ (0,0)\right] +2\left[ (1,2\sqrt{3})\right]+2\left[ (1,-2\sqrt{3})\right] -14\left[ (-2,0)\right] -18\left[ (2,0)\right].$ Then $\omega$ is new and general type. 
		\end{enumerate}  
	\end{examples}

	\begin{proof}[Proof of Theorem \ref{main-simple-poles}] The $1$-form $\omega$ is general type by \ref{key-proposition}. We show that there exists no pair $(Y,\eta)$ and a morphism $\phi$ such that $(X, \omega)\overset{\phi}{\longrightarrow}(Y, \eta)$ with $d:= \mathrm{deg}(\phi)\geq 2.$  Let $\tilde{g}$ be the genus of $Y.$ Then $\tilde{g} \leq g.$ and by the Riemann-Hurwitz formula \begin{equation*}
			2g-2= d(2\tilde{g}-2)+ \sum_{P\in X}(e_P-1).
		\end{equation*} Let $\left\lbrace P_{i}\right\rbrace _{i\in I}$ be the poles of $\omega$ with residues $\left\lbrace a_i\right\rbrace _{i\in I}$ and $\left\lbrace Q_{j}\right\rbrace _{j\in J}$ be the poles of $\eta$ with residues $\left\lbrace b_j\right\rbrace _{j\in J}$. Then by Lemma \ref{obv} (v), if $P_i\in \phi^{-1}(Q_j)$ for some $i\in I,j\in J$ then $a_i=e_{P_{i}}b_j$ where $e_{P_{i}}$ is the ramification index at $P_{i}.$
		
		We claim that there exists $i_1, i_2 \in I, i_1\neq i_2$ such that $P_{i_1}, P_{i_2} $ in the same fiber $\phi^{-1}(Q_j)$ for some $j\in J.$ Let say the claim is true. Then \begin{equation*}
			\frac{a_{i_1}}{a_{i_2}}=\frac{e_{P_{i_1}}b_j}{e_{P_{i_2}}b_j}=\frac{e_{P_{i_1}}}{e_{P_{i_2}}},
		\end{equation*}a contradiction on the assumptions on $\omega.$ Hence no such $\phi$ and $\eta$ exists. 
		
		Next, we shall prove the claim by contradiction. If the claim is not true, then for each pole $Q_j$ of $\eta,$ the fiber $\phi^{-1}(Q_j)$ is a pole $P_i$ of $\omega$ for some unique $j\in J$ and $i\in I.$ Then $e_{P_{i}}=d.$ Also  $\eta$ has at least $2g+2$ poles with nonzero residue. Otherwise, two or more poles of $\omega$ with nonzero residues will map to a single pole of $\eta.$ Then, the residues at the points in the fiber of a pole of $\eta$ will be $\mathbb{Q}-$linearly dependent, a contradiction. By Riemann-Hurwitz formula, the maximum possible number of points $P_i$ such that $e_{P_i}=d$ is $2g+2$ ($\tilde{g}=0$ with $d=2$). It follows that $\eta$ has at most $2g+2$ poles with nonzero residue; hence so does $\omega.$ But then the sum of these nonzero residues is zero, implying that they are dependent.
	\end{proof}
	\begin{remark}
		If $\omega_1,\dots, \omega_n$ are $1$-form on a curve $X$ satisfying the assumptions on Theorem \ref{main-simple-poles}, hence they are mew and general type $1$-form on $X.$ Then any $\mathbb{Q}-$linear combination $r_1\omega_1+\cdots+r_n\omega_n,$ $r_i\in\mathbb{Q},$ is also a new and general type $1$-form on $X.$ 
	\end{remark}
	\begin{example}[An example of new $1$-form on $\mathbb{P}^1$ by using Theorem \ref{main-simple-poles}] Consider the $1$-form on $\mathbb{P}^1$
		\begin{equation*}
			\omega= \frac{dx}{x-c_1}+\sqrt{2}\frac{dx}{x-c_2}+\sqrt{3}\frac{dx}{x-c_3}+xdx 
		\end{equation*} where $c_1, c_2, c_3 \in C$ are distinct points. Then $\omega$ is a new and general type $1$-form. In general  
		\begin{equation*}
			\omega= \sum_{i=1}^{n}c_i \frac{du_i}{u_i}+dv, \text{ }u_i, v\in C(x)
		\end{equation*} 
		is new and general type by Theorem \ref{main-simple-poles} if $u_i$  is of the form $a_i(x-b_i)^n$ for some $n\in \mathbb{Z}\setminus\{0\}$ and $c_i, c_j$ is not a rational multiple of each other for each $i, j$ with $ i\neq j$.  
	\end{example}

	\subsection{\boldmath {Existence of new and general type $1$-from in $\Omega(\mathrm{D})$}}
	Let $\mathrm{D}$ be an effective divisor on a curve $X$ and $\Omega(\mathrm{D}):=\left\lbrace \omega\in \Omega_X^1: \mathrm{div}(\omega)\geq-\mathrm{D}\right\rbrace .$ In this subsection, we investigate the existence of a new and general type $1$-form in $\Omega(\mathrm{D}).$  
	\begin{proof}[Proof of Corollary \ref{existence-in-omega-D}] The first part follows from Theorem \ref{main-simple-poles}. For the second part, by \cite[Problems VII.1]{MR95}, a $1$-form $\omega$ on $X$ exists with only a pole of order $2$ at any $P.$ We shall show that $(X,\omega)$ is new. If there is a pair $(Y,\eta)$ such that $(X, \omega)\overset{\phi}{\longrightarrow}(Y, \eta)$ with $\mathrm{deg}(\phi):=d\geq 2,$ then $\phi^{-1}(\phi(P))=\{P\}$ and $e_{P}=d\geq 2.$ But $e_{P}$ also divide $|\mathrm{ord}_P(\omega)+1|=1,$ a contradiction.
	\end{proof}           
	
	\begin{remark}
		In Corollary \ref{existence-in-omega-D} if $n+m\leq 2g+2,$ there may not exist a new and general type $1$-forms in $\Omega(\mathrm{D})$ (noted in Section \ref{Intro}). However, in some cases, the existence can be shown using Theorem \ref{main-higher-poles}. For example, let $\mathrm{D}=R_1+\mathfrak{I}(R_1)+2e_1+ 4e_2$ be an effective divisor on $E.$ Then $\Omega(D)$ contains a new and general type $1$-form. 
	\end{remark}
	
	For $\mathbb{P}^1,$ we have a complete answer on the existence of new and general type $1$-forms in $\Omega(\mathrm{D}).$ 
	\begin{proposition}
		Let $\mathrm{D}$ be an effective divisor on $\mathbb{P}^1$. There exists a new and general type $1$-form in $\Omega(\mathrm{D})$ for all $\mathrm{D}$ except $\mathrm{D}=R_1+R_2$ and $\mathrm{D}=wS, w\geq2.$  In the last case, a new $1$- form does exist.
	\end{proposition}
	\begin{proof}
		Let $\mathrm{D}=R_1+\cdots+R_n+w_1S_1+\cdots+w_mS_m, w_i\geq 2$ be an effective divisor on $\mathbb{P}^1.$ By Corollary \ref{existence-in-omega-D}, it is enough to assume $n+m\leq 2.$ We may assume that $\infty\notin \mathrm{Supp}(\mathrm{D}).$  We will divide the proof in the following cases:  
		
		Case (i): Let $n>0$ and $m=0.$ A $1$-form on $\mathbb{P}^1$ with only simple poles must have at least two poles. If $\omega\in \Omega^1_{\mathbb{P}^1}$ has only two simple poles, then $\omega=dx/x$ (up to a change of variable), which is exponential type.
		
		Case (ii): Let $n>0$ and $m> 0.$ There always exists a new and general type $1$-form, as the following example suggests. Let \begin{equation*}
			\omega= \frac{dx}{(x-S_1)^2(x-R_1)}
		\end{equation*}
		with $\mathrm{div}(\omega)= n\infty-R_1-2S_1.$
		It is a general type by Proposition \ref{key-proposition}. To show it is new, let $(\mathbb{P}^1, \omega)\overset{\phi}{\longrightarrow}(\mathbb{P}^1, \eta)$  and $\mathrm{deg}(\phi):=d\geq 2.$ Since $\omega$ has only one pole $S_1$ of order $\geq 2, \phi^{-1}(\phi(S_1))=\left\lbrace S_1\right\rbrace $ and $e_{S_1}=d\geq 2.$ But $e_{S_1}$ also divides $|\mathrm{ord}_{S_1}(\omega)+1|=1$ which implies $e_{S_1}=1,$ a contradiction.

		Case (iii): Let $n=0$ and $m>0.$ Let $m=2.$ If one of the $w_i$ is greater than $2,$ the existence follows from Theorem \ref{main-higher-poles}. Let $m=2$ and $w_1=w_2=2.$ New and general type $1$-form exists. Consider the $1$-form  \begin{equation*}
			\omega=\frac{dx}{(x-S_2)^2(x-S_3)^2}
		\end{equation*} with $\mathrm{div}(\omega)=2 \infty -2S_1 -2S_2.$ Then $\omega \in \Omega(D)$ and it is a general type by Proposition \ref{key-proposition}. Now if $(\mathbb{P}^1, \omega)\overset{\phi}{\longrightarrow}(\mathbb{P}^1, \eta)$  with $\mathrm{deg}(\phi):=d\geq 2,$ then $\eta$ is also general type. Hence $\eta$ has a zero, say at $Q$ and $\phi^{-1}(Q)=\left\lbrace \infty\right\rbrace,$ which implies   $e_{\infty}=d\geq 2.$ By equation (\ref{ram-ord}), $e_{\infty}(\mathrm{ord}_{Q}(\eta)+1)=3,$ which is not possible.

		Let $m=1.$ There are no new and general type $1$-forms in $\Omega(\mathrm{D}).$ If $\omega\in
		\Omega(\mathrm{D}),$ $\omega$ must have a pole at $S_1$ of order $n$ such that $n\leq w_1.$ Let \begin{equation*}
			\omega=\frac{(x-P_1)^{n_1}\cdots(x-P_r)^{n_r}dx}{(x-S_1)^n}:=g(x)dx.
		\end{equation*} Since $\mathrm{deg}(\omega)=-2,$ $n_1+\cdots+n_r=n-2.$
		Then the partial fraction of $g(x)$ will be 
		\begin{equation*}
			\sum_{i=2}^{n}\frac{c_i}{(x-S_1)^i}.
		\end{equation*} By Corollary \ref{rat-aut}, $\omega$ is exact type. Note that $\omega_1=\frac{c_1dx}{(x-S_1)^2}\in \Omega(\mathrm{D})$ is new.   
	\end{proof}

	Next, we calculate a lower bound on the dimension of the space of new and general type $1$-forms in $\Omega(\mathrm{D})$. We construct a class of such $1$-forms whose nonzero $C$-linear combinations are also new and of general type.  
	\begin{lemma}\label{lemma-dim-count} Let $\mathrm{D}=R_1+\cdots+R_n+w_1S_1+\cdots+w_mS_m$ with $w_i\geq 2,$ be an effective divisor on $X.$ Let $\omega$ be a $1$-form having at least one simple pole, at least two  poles of order $\geq 2$ with nonzero residues and polar divisor $\mathrm{D}.$ If $w_i-1\in \mathfrak{D}(2g_X+n+m+1), 1\leq i\leq m,$ there is an  $i$ such that $w_i\neq w_j$ for $1\leq j\neq i\leq m$ and two nonzero residues of $\omega$ at the poles of order $\geq 2$ are linearly independent over $\mathbb{Q},$ then $\omega$ is new and general type.
		\begin{proof} By Proposition \ref{key-proposition}, the $1$-form $\omega$ is general type . We will show that $\omega$ is new. If not, let $(X, \omega)\overset{\phi}{\longrightarrow}(Y, \eta)$ with $\mathrm{deg}(\phi):=d\geq 2.$ If $g_Y\geq 1,$ using the same argument as in Lemma \ref{main-lemma} (ii), one can show a contradiction. Let $g_Y=0.$ First, let $\eta$ has at least two poles of order $\geq 2.$  Then by Lemma \ref{obv} (ii) for each $k$ there exists an $l, l\neq k, 1\leq l\leq m$ such that $S_k\notin \phi^{-1}(\phi(S_l)).$ By Riemann-Hurwitz formula, for any $S_k$\begin{equation*}
				(e_{S_k}-1)+\sum_{ R_{i}\in \phi^{-1}(\phi(R_1))}(e_{R_i}-1)+\sum_{S_i\in\phi^{-1}(\phi(S_l))}(e_p-1)\leq \sum_{P\in \mathbb{P}^1}(e_P-1)=2d-2,
			\end{equation*}implies $e_{S_k}\leq n+m+1.$ Then $e_{S_k}=1$ and for each $i$ there exists $j, 1\leq i\neq j\leq m$ such that $w_i=w_j,$ a contradiction.  
			
			Let $\eta$ has only one pole at $Q$ of order $\geq 2.$ Then inverse image of $Q$ contains all the poles of order $\geq 2$ of $\omega,$ which will contradict the assumption that $\omega$ has two $\mathbb{Q}-$linearly independent residues at the poles of order $\geq 2.$ 
		\end{proof} 
	\end{lemma}

	Now the question is for which effective divisors on $\mathbb{P}^1,$ elliptic or hyperelliptic curve, one can construct such a  $1$-form satisfying the conditions of Lemma \ref{lemma-dim-count}. We have the following answers:  \begin{enumerate}
		\item On $\mathbb{P}^1,$ let $\mathrm{D}=R_1+\cdots+R_n+z_0 \infty + w_1S_1+\cdots+w_mS_m,$ where $w_i\geq 2, z_0\geq1.$   Consider the $1$-form \begin{equation*}\omega=\left( cx^{z_0-2}+\sum_{i=1}^{n}\frac{c_i}{x-R_i}+\sum_{j=1}^{w_1}\frac{d_{1j}}{(x-S_1)^j}+\cdots+\sum_{j=1}^{w_m}\frac{d_{mj}}{(x-S_m)^j}\right)dx.
		\end{equation*}
		The polar divisor of $\omega$ is $\mathrm{D}$ with $\mathrm{Res_{S_i}(\omega)}=d_{i1}.$  
		
		\item On $E,$ let $\mathrm{D}=R_1+\mathfrak{I}(R_1)+\cdots+R_n+\mathfrak{I}(R_n)+w_1(S_1+\mathfrak{I}(S_1))+\cdots+w_m(S_m+\mathfrak{I}(S_m))+v_1T_1+v_2T_2+v_3T_3+v_4\infty,$ where all $w_i\geq 2, v_1, v_2, v_3, v_4$ not all zero even positive integer. Consider the $1$-form \begin{equation*} \omega=\left(c(x-x_P)^u+ \sum_{i=1}^{n}\frac{c_i}{x-x_{R_{i}}}+\sum_{i=1}^{m}\sum_{j=1}^{w_i}\frac{d_{ij}}{(x-x_{S_i})^j}+\sum_{i=1}^{3}\sum_{j=1}^{k_i}\frac{f_{ij}}{(x-e_{i})^j}\right)\frac{dx}{y}
		\end{equation*}where $P_i=(x_{P_i},y_{P_i}), R_i=(x_{R_i},y_{R_i}), S_i=(x_{S_i},y_{S_i}).$ The polar divisor of $\omega$ is $\mathrm{D}$ with $v_i=2k_i,$ for $i=1, 2, 3$ and $v_4=2u,$ and $\mathrm{Res_{S_i}(\omega)}=d_{i1}.$

		\item On $H,$ let $\mathrm{D}=R_1+\mathfrak{I}(R_1)+\cdots+R_n+\mathfrak{I}(R_n)+w_1(S_1+\mathfrak{I}(S_1))+\cdots+w_m(S_m+\mathfrak{I}(S_m))+\sum_{i=1}^{2g+1}v_ie_i+v_{2g+2}\infty$ where all $w_i\geq 2,$ and $v_i$ not all zero even positive integer. Consider the $1$-form \begin{equation*} \omega=\left(c(x-x_P)^u+ \sum_{i=1}^{n}\frac{c_i}{x-x_{R_{i}}}+\sum_{i=1}^{m}\sum_{j=1}^{w_i}\frac{d_{ij}}{(x-x_{S_i})^j}+\sum_{i=1}^{2g+1}\sum_{j=1}^{k_i}\frac{f_{ij}}{(x-e_{i})^j}\right)\frac{dx}{y} \end{equation*}
		where $w_i\geq 2, P_i=(x_{P_i},y_{P_i}), R_i=(x_{R_i},y_{R_i}), S_i=(x_{S_i},y_{S_i}).$ The polar divisor of $\omega$ is $\mathrm{D}$  with $v_i=2k_i$ for $i=1,\dots, 2g+1$ and $v_{2g+2}=2u+2g-2$ with $\mathrm{Res_{S_i}(\omega)}=d_{i1}.$         
	\end{enumerate}

	\begin{theorem}\begin{enumerate}[(i)]\label{dim-thm}
			\item Let $\mathrm{D}=R_1+\cdots+R_n+z_0\infty +w_1S_1+\cdots+w_mS_m$ with $n\geq 1, m\geq 2, z_0\geq 1, w_i\geq 2,$ be an effective divisor on $\mathbb{P}^1.$ Let $s$ be the third  element in the ordered set $\mathfrak{D}(n+m+1)$ and $w_i\geq s.$ Then there exists a $C-$vector space $W\subseteq \Omega(\mathrm{D})$ of dimension at least $n \lfloor \frac{m}{2} \rfloor$ such that any nonzero element of $W$ is new and general type.
			\item Let $\mathrm{D}=R_1+\mathfrak{I}(R_1)+\cdots+R_n+\mathfrak{I}(R_n)+w_1(S_1+\mathfrak{I}(S_1))+\cdots+w_m(S_m+\mathfrak{I}(S_m))+v_1e_1+v_2e_2+v_3e_3+v_4\infty,$ where all $w_i\geq 2, n\geq1, m\geq 2$ and $v_1, v_2, v_3, v_4$ not all zero, be an effective divisor on the elliptic curve $E.$ Let $s$ be the third element in the ordered set $\mathfrak{D}(n+m+1)$ and $w_i\geq s, v_i\geq s.$ Then there exists a $C-$vector space $W\subseteq \Omega(\mathrm{D})$ of dimension at least $n \lfloor \frac{m}{2} \rfloor$ such that any nonzero element of $W$ is new and general type.
			\item Let $\mathrm{D}=R_1+\mathfrak{I}(R_1)+\cdots+R_n+\mathfrak{I}(R_n)+w_1(S_1+\mathfrak{I}(S_1))+\cdots+w_m(S_m+\mathfrak{I}(S_m))+\sum_{i=1}^{2g+1}v_ie_i+v_{2g+2}\infty$ where all $w_i\geq 2, n\geq 1, m\geq 2$ and $v_i$ not all zero,  be an effective divisor on a hyperelliptic curve $H.$ Let $s$ be the third element in the ordered set $\mathfrak{D}(n+m+1)$ and $w_i\geq s, v_i\geq s.$ Then there exists a $C-$vector space $W\subseteq \Omega(\mathrm{D})$ of dimension at least $n \lfloor \frac{m}{2} \rfloor$ such that any nonzero element of $W$ is new and general type.
		\end{enumerate}
		\begin{proof} As the idea of the proof is similar for all parts of the theorem, we only write the proof of (i). First, construct a $1$-form $\omega\in \Omega(\mathrm{D})$ with polar divisor $\infty+R_1+ v_1S_1+v_2S_2, v_1, v_2\in \mathfrak{D}(n+m+1), v_1\neq v_2, v_1\leq s, v_2\leq s$ and $\mathbb{Q}$- linearly independent residues at $S_1, S_2.$  Since $\mathfrak{D}(n+m+1)\subseteq \mathfrak{D}(4),$ by Lemma \ref{lemma-dim-count}, $\omega$ is new and general type. For any nonzero $c\in C,$  $c\omega_1$ is also new and general type.  Next, construct another new and general type $1$-form $\eta$ same way as $\omega$ but with polar divisor $\infty+R_2+v_3S_3+v_4S_4.$ For any nonzero $c, d \in C,$ let $\theta=c\omega+d\eta.$ Then $\theta\in \Omega(\mathrm{D})$ and the polar divisor of $\theta$ is $$\infty+R_1+R_2+v_1S_1+v_2S_2+v_3S_3+v_4S_4.$$
			Since residue is a $C$-linear map, $\mathrm{Res_{S_i}(\theta)}=c_1\mathrm{Res_{S_i}(\omega)}$ for $i=1,2$ and $\mathrm{Res_{S_i}(\theta)}=c_2\mathrm{Res_{S_i}(\eta)}$ for $i=3,4.$ Then $\theta$ satisfies the conditions of Lemma \ref{lemma-dim-count}. Hence, $\theta$ is new and general type. Observe that $\omega_1$ and $\omega_2$ are linearly independent over $C.$ 
			
			Note that in the construction of $\eta,$ if one chooses $R_1$ instead  of $R_2$ and constructs another $1$-form $\tilde{\eta},$ then all of the above arguments are true. Using these ideas, one can construct $C-$linearly independent new and general type $1$-forms in $\Omega(D).$  The number of such $1$-forms are at least  $n\lfloor \frac{m}{2}\rfloor.$   	
		\end{proof} 
	\end{theorem}

	\section{Algorithm and Its Applications}
	\subsection{An Algorithm.} Let $(X, \omega)$ be a pair. Determining the place of $(X, \omega)$ in the classification of first order differential equations is an interesting algorithmic question.  When $X=\mathbb{P}^1,$ Corollary \ref{rat-aut}, provides a complete algorithm to determine where $1$-form on $X$ is of exact, exponential or general type.   We are not aware of an algorithm that answers the question in full generality. However, there are algorithms that can decide whether an equation is of exact or of exponential type. The first one is an application of the Coates algorithm (\cite[Algorithm 2.11]{NNvdP15}), and the second one follows from the work of Baldassarri and Dwork (\cite[Section 6]{BD79}). Sometimes Proposition \ref{key-proposition} is helpful.

	For the rest of this subsection, we assume that $C$ is the field of complex numbers. We produce an algorithm that decides whether a $1$-form on $\mathbb{P}^1$ is new or old. The main idea is that given a $1$-form $\omega$ on $\mathbb{P}^1,$ if there exists a morphism $\phi:\mathbb{P}^1\rightarrow\mathbb{P}^1$ and a $1$-form $\eta$ on $\mathbb{P}^1$ such that  $\phi^*\eta=\omega$ then the possibilities of branch data of $\phi$ at the zeros and poles of $\omega$ is determined by the order of zeros and poles of $\omega$ (see Lemma \ref{obv}). All the ramified points of $\phi$ may not come from the set of zeros and poles of $\omega.$ But those ramified points outside the zeros and poles of $\omega$ will not take part in finding a suitable $\eta.$ Hence, one can find a set of possible abstract branch data for $\phi.$ That's where the Hurwitz realization problem comes into play and decides whether there exists a meromorphic function on $\mathbb{P}^1$ with given abstract branch data. As we know in Section \ref{hur-prob}, the Hurwitz realization problem for the Riemann sphere is not fully solved. Hence, our algorithm is only applicable to the following $1$-forms on $\mathbb{P}^1$ that satisfy one of the following: \begin{enumerate}[(i)]	\item the number of zeros of $\omega$ is less than equal to $2,$
		\item the number of simple poles of  $\omega$ is less than  equal to $2,$
		\item the number of poles of order $\ge 2$ of  $\omega$ is less than equal to $2,$
		\item $\frac{1}{2}(\mathrm{deg}(\mathrm{div}_0(\omega))+m)\leq 5,$ where $\mathrm{div}_0(\omega)$ denotes the divisor of zeros of $\omega$ and $m$ is the number of zeros of $\omega.$ 
	\end{enumerate}

	Note that if $\omega$ is of exact type, then whether it is new or old can be determined by Corollary~\ref{rat-aut}. If $\omega$ is of exponential type, then it is old. Therefore, in order to decide whether a $1$-form on $\mathbb{P}^1$ is new or old, it suffices to restrict attention to $1$-forms of general type.  
	\begin{algorithm}

		\hrule

		Input: A general type $1$-form $\omega$ on $\mathbb{P}^1$ from the applicable cases.
		\vspace{2mm}\\
		Output: Decide whether $\omega$ is new or old. If it is old, determine a morphism $\phi:\mathbb{P}^1\rightarrow \mathbb{P}^1$ and $1$-form $\eta$ on $\mathbb{P}^1$ such that $\phi^*\eta=\omega.$
		
		\begin{enumerate}[1)]
			\item Compute all the zeros, poles and residues of $\omega.$ 
			\item Calculate the possible  degrees of a possible morphism $\phi$ by Proposition \ref{deg-bound}. If $\omega$ is general type and has either a single zero or a single pole of order $\geq 2,$ say at $P,$ then $d$ divides  $|\mathrm{ord}_P(\omega)+1|.$ In that case, $d=|\mathrm{ord}_P(\omega)+1|$ can be excluded. 
			\item Fix a possible degree of $\phi$.   
			\item  Using Lemma \ref{obv}, calculate the possible ramification index at all the zeros and poles of order $\geq
			2.$ The Ramification index at the simple poles is bounded by degree.
			\item From the above data and the Riemann-Hurwitz formula, create a possible abstract branch data, keeping in mind that ramified points of $\phi$ may come from outside the zeros and poles of $\omega.$ As $\omega$ is from applicable cases, the abstract branch data will be of the type for which the Hurwitz problem has a solution.    
			\item If there is branch data, then calculate the zeros and poles of a possible $1$-form $\eta.$
			\item If there is no branch data or there is branch data with no possible $1$-form, repeat the same process for all possible degrees.
			\item If there is branch data and a $1$-form $\eta,$ and the branch data is realizable, then $\omega$ is old. Otherwise,  $\omega$ is new.         
		\end{enumerate}
		\hrule
	\end{algorithm}

	Consider the differential equations of the form $u'=g(u)$ where either $g(u)$ or $1/g(u)$ is a Laurent polynomial. The corresponding pair is $(\mathbb{P}^1, \omega)$ where $\omega=dx/g(x).$ One can completely determine its place in the classification of first order differential equations. They can't be of Weierstrass type. Whether it is of exact, of exponential or of general type can be decided by Corollary \ref{rat-aut}. As $g(u)$ or $\frac{1}{g(u)}$ is a Laurent polynomial, $\omega$ can have at most two zeros, two simple poles, or two poles of order $\geq 2.$ Hence, the above algorithm is applicable for the above pair.   
	\subsection{Applications} We give a few examples that demonstrate how the algorithm works. For a small number of zeros and poles the algorithm can be performed by hand calculation.

	\begin{example}
		The differential equation
		$y'=y^2(y-1)^3(y-2)^5$ is new and of general type.

		The pair associated to the equation is $(\mathbb{P}^1, \omega)$ where \begin{equation*}
			\omega=\frac{dx}{x^2(x-1)^3(x-2)^5}
		\end{equation*} with
		\begin{equation*}
			\mathrm{div}(\omega)= 8\left[ \infty\right] -2\left[ 0\right] -3\left[ 1\right] -5\left[ 2\right] .
		\end{equation*}
		By Proposition \ref{key-proposition}, $\omega$ is general type . If $(\mathbb{P}^1, \omega)\overset{\phi}{\longrightarrow}(\mathbb{P}^1, \eta)$ then $(\mathbb{P}^1,\eta)$ is also of general type. Since $\omega$ has only one zero at $\infty$, the degree $d:=\mathrm{deg}(\phi)$ divides $9.$ The only possible $d$ is $3.$ By the Riemann-Hurwitz formula, 
		$4=\sum_{P\in \mathbb{P}^1}(e_P-1).$ Also
		$3=\sum_{P\in\phi^{-1}(\phi(P))}e_P.$ Note that $e_{\infty}=d=3.$
		Consider the following table for calculations:
		
		\hspace{5.5cm}\begin{tabular}{|p{1cm}||p{1.5cm}|p{1.4cm}|}
			\hline
			\multicolumn{3}{|c|}{Table $1$} \\
			\hline
			$P$& $e_P$& $\frac{\mathrm{ord_P(\omega)+1}}{e_P}$\\
			\hline
			$0$& $1$& $-1$\\
			\hline
			$1$&$1,2$& $-2,-1$\\
			\hline
			$5$&$1,2$& $-4,-2$\\
			\hline
		\end{tabular}\\

		For $d=3,$ possible branch data are $\{\{2,1\},\{2,1\},\{2,1\},\{2,1\}\}, \{\{3\},\{2,1\},\{2,1\}\}$ and $\{\{3\},\{3\}\}.$ As $e_{\infty}$ is $3,$ the first case is not possible. $\{\{3\},\{3\}\}$ is not possible as ramification index at $0$, $1$ and $2$ cannot be $3,$ so the other ramification point cannot come from $\left\lbrace 0,1,2\right\rbrace .$ If it comes from outside of  $\left\lbrace 0,1,2\right\rbrace $ then all $e_0,e_1,e_2$ is $1.$ They should be in the fiber of the only pole of $\eta,$ which is not possible. To get  $\{\{3\},\{2,1\},\{2,1\}\},$ we need two ramified point of ramification index $2$ and two unramified point form the set $\left\lbrace0,1,2 \right\rbrace$, which is not possible. It cannot come from outside that set because inverse image of poles of $\eta$ are poles of $\omega.$  Hence $\omega$ is new.          
	\end{example}
	
	\begin{example}
		The equation $y'= \frac{1}{2}(y^5-y^3)$ old and of general type.
		
		The pair corresponding to the differential equation is $(\mathbb{P}^1, \omega)$ where
		\begin{equation*}
			\omega=\frac{2dx}{x^3(x^2-1)}\text{ }\text{with}\text{ } \mathrm{div}(\omega)= 3\left[ \infty\right] -3\left[ 0\right] -\left[ i\right] -\left[ -i\right].
		\end{equation*} 
		
		By Proposition \ref{key-proposition}, $\omega$ is general type. Since $\omega$ has only one zero of order $3$ at $\infty$ and only one pole of order $\geq 2$ at $0$ we have $e_{\infty}=d=e_{0}.$ As $e_0$ divides $|\mathrm{ord_0(\omega)+1}|=2.$ The only possible $d$ is $2.$ We shall check if  $(\mathbb{P}^1, \omega)\overset{\phi}{\longrightarrow}(\mathbb{P}^1, \eta)$ with $\mathrm{deg}(\phi)=2.$   
		
		Here $e_{\infty}=e_{0}=2.$ For a degree $2$ map, only branch data is $\{\{2\},\{2\}\}.$ Hence $e_{i}=e_{-i}=1.$ Then $e_{\infty}=e_{0}=2,e_{i}=e_{-i}=1,$ $\phi^{-1}(\phi(\infty))=\left\lbrace\infty \right\rbrace, \phi^{-1}(\phi(0))=\left\lbrace 0 \right\rbrace$ and $\phi(i)=\phi(-i).$ Hence $\mathrm{ord_\infty(\eta)}= 1, \mathrm{ord_0(\eta)}=-2$ and $\mathrm{ord_{\phi(i)=\phi(-i)}(\eta)}=-1$ and $\mathrm{deg}((\mathrm{div})(\eta))=-2.$ If we assume $\phi(\infty)=\infty, \phi(0)=0$ and $\phi(i)=\phi(-i)=1$ then $\eta=\frac{dx}{x^3-x^2}$ and $\phi(x)=x^2.$ The equation $y'= \frac{1}{2}(y^5-y^3)$ is old.    
	\end{example}

	\begin{example} Consider the differential equation $(\mathbb{P}^1,\omega)$ where \begin{equation*}
			\omega=\frac{(y-a_1)^2(y-a_2)(y-a_3)}{(y-a_4)(y-a_5)(y-a_6)^2(y-a_7)^2} \text{, }\text{$a_i\in \mathbb{C}.$}
		\end{equation*}
		
		The $1$-form $\omega$ is general type by Proposition \ref{key-proposition} and $$\mathrm{div}(\omega)= 2\left[ a_1\right] +\left[ a_2\right] +\left[ a_3\right] -\left[ a_4\right] -\left[ a_5\right] -2\left[ a_6\right] -2\left[ a_7\right] .$$   
		If $(\mathbb{P}^1, \omega)\overset{\phi}{\longrightarrow}(\mathbb{P}^1, \eta)$ with $d:=\mathrm{deg}(\phi)\geq 2$ then by Proposition \ref{deg-bound}, possible $d$'s are $2$ and $3.$ Consider the following table for calculations:
		
		\hspace{5.5cm}\begin{tabular}{|p{1cm}||p{1.5cm}|p{1.4cm}|}
			\hline
			\multicolumn{3}{|c|}{Table $2$} \\
			\hline
			$P$& $e_P$& $\frac{\mathrm{ord_P(\omega)+1}}{e_P}$\\
			\hline
			$a_1$& $1,3$& $3,1$\\
			\hline
			$a_2$&$1,2$& $2,1$\\
			\hline
			$a_3$&$1,2$& $2,1$\\
			\hline
			$a_6$&$1$ &$ -2$\\
			\hline
			$a_7$&$1$&$-2$\\
			\hline
		\end{tabular}\\

		Observe that $a_6$ and $a_7$ are unramified points and by Lemma \ref{obv} (ii),  $a_6, a_7$ maps to a pole of order $2.$ Also all the $e_{a_1}, e_{a_2}, e_{a_3}$ can not be $1$ and $a_1, a_2,a_3$ can not be in the same fiber.  
		
		Case d=2: The only branch data for a degree $2$ map is $\{\{2\},\{2\}\}.$ Then  $e_{a_1}=1$ and $\phi^{-1}(\phi(a_1))$ must contain another zero of $\omega.$ As $d=2,$ that zero must be an unramified point. This is not possible because $ \mathrm{ord_{a_1}(\omega)}\neq \mathrm{ord_{a_2}(\omega)}=\mathrm{ord_{a_2}(\omega)}.$

		Case d=3: Here $e_{a_1}$ must be $3.$ Otherwise, $e_{a_1}=1$ and $\phi(a_1)$ is zero of $\eta,$ the inverse image $\phi^{-1}(\phi(a_1))$ must contain either two unramified zeros or one ramified zero with ramification index $2,$ which is not possible. We shall now look for possible ramification index at $a_2$ and $a_3.$ The $1$-form $\eta$ is a general type; hence, must have a zero. That implies both $e_{a_2}$ and $e_{a_3}$ cannot be $2.$ Let $e_{a_2}=1.$ Then $\phi(a_2)$ is a zero of $\eta,$ and the inverse image $\phi^{-1}(\phi(a_2))$  must contain either two unramified zeros or one ramified zero with ramification index $2,$ which is not possible.  
	\end{example}

	There are some explicit criteria for new  $1$-forms in a few circumstances. Regarding this, we have the following results:

	\begin{proposition}\label{one-zero-gen}
		Let $\omega$ be a general type meromorphic $1$-form on a curve $X.$ If $\omega$ has only one zero of order $m$ such that $m+1$ is prime, then $\omega$ is new.
		\begin{proof} If $\omega$ is not new, let $(X, \omega)\overset{\phi}{\longrightarrow}(Y, \eta)$ with $d:=\mathrm{deg}(\phi)\geq 2.$ Let $P$ be the only zero of $\omega$. Since $\eta$ is a general type, it has a zero at $Q$ of order $r\geq1.$ By Lemma \ref{obv}, $\phi^{-1}(Q)=\left\lbrace P\right\rbrace.$ Then $e_{P}=d$ and $d (r+1)=m+1,$ a contradiction.
		\end{proof}	
	\end{proposition}
	
	\begin{example}
		By the proposition \ref{one-zero-gen},\begin{enumerate}[(i)] \item The equations $y'=y^3-y^2$ and $y'=y/(y+1)$ are new and of general type.
			\item The equation $y'=a_2y^2+\dots+a_ny^n, a_i\in C$ with $(a_2,a_3)\neq (0,0),$ is of general type (\cite[Example 7.4]{KRS24}). It is new if $n-1$ is prime.   
			\item The equation $y'=y^n-1$ is new and of general if $n\geq 3$ and $n-1$ is prime. 
		\end{enumerate} 
	\end{example}
	\begin{proposition}\label{two-zero-gen}
		Let $\omega$ be a general type  meromorphic $1$-form on a curve $X$ having only two zeros, one of order $1$ and the other of order $m$ such that both $m+1$ and $m+3$ are prime, then $\omega$ is new.
		\begin{proof}
			If $\omega$ is not new, let $(X, \omega)\overset{\phi}{\longrightarrow}(Y, \eta)$ with $d:=\mathrm{deg}(\phi)\geq 2.$ Let $P_1, P_2$ be two zeros of $\omega$ of order $1$ and $m$ respectively. As $\eta$ is a general type, it has a zero at $Q$ of order $r\geq 1.$ Three cases could occur: (i) $\phi^{-1}(Q)=\left\lbrace P_1\right\rbrace,$ (ii) $\phi^{-1}(Q)=\left\lbrace P_2\right\rbrace,$ (iii) $\phi^{-1}(Q)=\left\lbrace P_1, P_2\right\rbrace.$ In the 
			first and second cases, we have  $d(r+1)=2$ and $d(r+1)=m+1$ respectively, which is absurd. In the last case, we get 
			$(e_{P_1}+e_{P_2})(r+1)= m+3,$ a contradiction.
		\end{proof}
	\end{proposition}
	\begin{example}
		Consider the equation 
		\begin{equation*}
			y'=(y-R)\prod_{(i,j)\in I\times J}(y-S_{ij})^{m_{ij}}
		\end{equation*}
		where $I, J$ are finite index set and $R, S_{ij}\in C$ are distinct and $m_{ij}\geq 2.$ Let the set $\left\lbrace m_{ij}\right\rbrace_{(i,j)\in I\times J} :=\left\lbrace n_1,\dots,n_k\right\rbrace $ where $n_r=m_{ij}$ for some $i,j$ and  $n_i$ appears $r_i-$times in the equation. Let $s$ be a proper divisor of $\sum_{i,j}m_{ij}$ except $1$ and $\sum_{i,j}m_{ij}.$ The equation is new and of general type if $s$ does not divide all $r_i.$

		\begin{proof}
			The pair associated to the equation is $(\mathbb{P}^1, \omega)$ with  $$\mathrm{div}(\omega)=\left( \sum m_{ij}-1\right) \infty - R -\sum m_{ij}S_{ij} .$$ By Proposition \ref{key-proposition}, $\omega$ is general type. Let  $(\mathbb{P}^1, \omega)\overset{\phi}{\longrightarrow}(\mathbb{P}^1, \eta)$ with $d:=\mathrm{deg}(\phi)\geq 2.$ As  $\omega$ is general type so is $\eta.$ Let $Q$ be a zero of $\eta.$ Then $\phi^{-1}(Q)=\{\infty\}$ with $e_{\infty}=d,$ and $d$ divides $\mathrm{ord_{\infty}(\omega)+1}=\sum_{i, j}m_{ij}.$ The possible degrees of $\phi$ are divisors of \ $\sum_{i,j}m_{ij}$ except $1$ and $\sum_{i,j}m_{ij}.$ Also $e_{R}=d$ with $\phi^{-1}(\phi(R))=\left\lbrace R\right\rbrace .$  By Riemann-Hurwitz formula, $e_{S_{ij}}=1$ for all $i, j.$ Let $T$ be a pole of order $\geq 2$ of $\eta.$ Then by Lemma \ref{obv}, any two poles of $\omega$ in    $\phi^{-1}(T)$ must have same orders. Let $s$ be a divisor of $\sum_{i,j}m_{ij}, s\neq 1, s\neq \sum_{i,j}m_{ij}$ i.e. $s$ is a possible degree of $\phi.$  Each fiber of a pole of $\eta$ of order $\geq 2$ contains $s$ elements and thus $s$ divides every $r_i,$ a contradiction. Therefore, $\omega$ is new and general type.    
		\end{proof}    
	\end{example}

	\bibliographystyle{abbrv}
	\bibliography{NGADE}
\end{document}